\documentclass[a4paper,11pt]{article}
\usepackage[margin=1in]{geometry}
\usepackage{amsmath,amsthm,amssymb}
 \usepackage{enumitem,color}
 \usepackage{hyperref}
\hypersetup{
 colorlinks=true,
 citecolor=blue,
 linkcolor=black,
 filecolor=magenta,
 urlcolor=cyan
}
\usepackage[capitalize]{cleveref}
\usepackage{xcolor}
\usepackage{graphics,graphicx}

\newtheorem{theorem}{Theorem}
\newtheorem{lemma}{Lemma}
\newtheorem{corollary}{Corollary}
\newtheorem{rlt}{Theorem}

\newtheorem{rll}{Lemma}

\newtheorem{problem}{Problem}
\newtheorem{claim}{}[lemma]
\newtheorem{claim1}{Claim}
\newcommand*\sq{\mathbin{\vcenter{\hbox{\rule{0.75ex}{1.0ex}}}}}

\begin{document}

\title{Vertical critical co-gem-free graphs}

\author{ Manoj Belavadi\thanks{Computer Science Unit, Indian Statistical
Institute, Chennai Centre, Chennai 600029, India. Supported by Anusandhan National Research Foundation (ANRF), Govt. of India, under the National Postdoctoral Fellowship (NPDF) program (No. PDF/2025/005208). ORCID: 0000-0002-3153-2339.} \and  T.~Karthick\thanks{ Computer Science Unit, Indian Statistical
Institute, Chennai Centre, Chennai 600029, India. Email: karthick@isichennai.res.in. ORCID: 0000-0002-5950-9093.   } }

\date{\today}
\maketitle

\begin{abstract}
Given a graph $G$, let $\chi(G)$ denote the chromatic number of $G$. For $k\in \mathbb{N}$, a graph $G$ is \emph{$k$-vertex-critical} if $\chi(G)=k$ and $\chi(G-v)< k$ for all $v\in V(G)$.   A recent problem of Beaton and Cameron [TCS 1042 (2025) 115234] asks   for which graphs $H$ of order five, are there finitely many $k$-vertex-critical (co-gem, $H$)-free graphs, for all $k\in \mathbb{N}$? Here we identify three distinct graphs on five vertices that yield an affirmative  answer to this problem. More precisely,  we show that  for each $k\in \mathbb{N}$, there are finitely many $k$-vertex-critical (co-gem, $H$)-free graphs, where $H\in \{$paraglider, dart, house$\}$, by analysing the structure of such graphs.   Our results together with a result of Couturier et al. [Algorithmica 71:1 (2015) 21--35] imply that for each $k\in \mathbb{N}$, there is a polynomial-time certifying algorithm for  $k$-{\sc Coloring} of (co-gem, $H$)-free graphs, where $H\in \{$paraglider, dart, house$\}$.
\end{abstract}

\section{Introduction}

    All graphs considered in this paper are  simple and undirected. For missing notation and terminologies, we refer to \cite{BM-book}.  For a graph $G$, we use $V(G)$ and $E(G)$ to denote the vertex-set and the edge-set of $G$, respectively.  We say that a graph $G$ \emph{contains} a graph $H$, if $H$ is an induced subgraph of $G$.  Given a graph $H$, a graph  is \emph{$H$-free} if it does not contain $H$. For a given set of graphs $\{H_1,H_2,\ldots, H_t\}$, a graph is \emph{$(H_1,H_2,\ldots, H_t)$-free} if it is $H_i$-free for all $i\in \{1,2,\ldots, t\}$. Also for a given set of graphs $\mathcal{H}$, a graph is {\em $\mathcal{H}$-free} if it does not contain any graph from $\mathcal{H}$.   A \emph{stable set} (resp., \emph{clique}) in a graph  $G$ is a set of mutually nonadjacent (resp., adjacent) vertices in $G$.

     A $k$-\emph{coloring} of a graph $G$ is a function  $f: V(G) \rightarrow \{1,2,\ldots, k\}$ such that $f(u)\neq f(v)$, for every edge $uv\in E(G)$.  In other words, a $k$-coloring of $G$ is a partition of $V(G)$ into $k$ stable sets.  The smallest integer $k$ for which $G$ admits a $k$-coloring is called the \emph{chromatic number} of $G$, and is denoted by $\chi(G)$.   A graph $G$ is $k$-colorable if $\chi(G)\leq k$, and is \emph{$k$-chromatic} if $\chi(G)=k$.
     For $k\in \mathbb{N}$, a graph $G$ is \emph{$k$-vertex-critical} if $\chi(G)=k$ and $\chi(G-v)< k$ for all $v\in V(G)$.
A graph $G$ is {\em perfect} if for every  induced subgraph  $H$ of $G$, $\chi(H)$ equals the size of a largest clique of $H$.

For a fixed integer $k$, the  $k$-{\sc Coloring} problem asks whether the given graph is $k$-colorable. It  is well-known to be \textsf{NP}-complete. This motivated several researchers to study $k$-{\sc Coloring} for restricted graph classes, and we refer to \cite{survey} for an extensive survey.
   An algorithm is said to be \emph{certifying} if it returns a certificate of correctness, for each output, that can be verified in polynomial time. A certifying algorithm for  $k$-{\sc Coloring}  can output a $k$-coloring of $G$ when it is an `yes' instance or output a ($k$+1)-vertex-critical induced subgraph of $G$ when it is a `no' instance.  The polynomial-time algorithm  for  $k$-{\sc Coloring} of   $P_5$-free graphs given in \cite{HoaKam2010}, and the polynomial-time algorithm for $k$-{\sc Coloring} of   ($P_5+\ell P_1$)-free graphs given in \cite{CGKP2015} are not certifying.  Interestingly, we have the following.

    \begin{rlt}[Folklore] \label{thm:finite-imp-P}
     Given a set of graphs $\cal{H}$, let $\cal{G}$ be  the class of $\cal{H}$-free graphs.  For a fixed $k\in \mathbb{N}$, if there are finitely many $(k+1)$-vertex-critical graphs in $\cal{G}$, then there is a polynomial-time  algorithm for $k$-{\sc Coloring} of $\cal{G}$.
    \end{rlt}
   Using \cref{thm:finite-imp-P}, for a fixed  $k\in \mathbb{N}$, one can obtain a polynomial-time  certifying  algorithm for   $k$-{\sc Coloring} of  some restricted graph classes.  Thus  we have the following  problem.

    \begin{problem}\label{prob-fin}
      Given a set of graphs $\cal{H}$, let $\cal{G}$ be  the class of $\cal{H}$-free graphs.  For a fixed $k\in \mathbb{N}$, are there   finitely many  $k$-vertex-critical graphs in $\mathcal{G}$?
    \end{problem}

 \cref{prob-fin} has gathered much attention over the past few decades and remains of current interest.
 Recall that the class of co-graphs (or the class of $P_4$-free graphs) are perfect, and hence  the only $k$-vertex-critical co-graph is the complete graph on $k$ vertices. For a graph $H$ and for $k\le 4$, \cref{prob-fin} has been completely solved for $H$-free graphs by Chudnovsky et  al. \cite{CGSZ2020,CGSZ2021}. They proved that there are finitely many $4$-vertex-critical $H$-free graphs if and only if $H$ is an induced subgraph of one of the graphs in \{2$P_3$, $P_6$, $P_4$+$r P_1$\},  where $r\ge  0$.    It is also known that  there are finitely many $k$-vertex-critical ($P_3+rP_1$)-free graphs for all $k\geq 1$ and $r\geq 0$ \cite{ACHS2024}, but there are infinitely many $k$-vertex-critical $2P_2$-free graphs for all $k\geq 5$, and hence there are infinitely many $k$-vertex-critical $P_5$-free graphs for all $k\geq 5$  \cite{HoaMoo2015}.
         This prompted several researchers to study  $k$-vertex critical ($P_t$, $H$)-free graphs for $t\geq 5$ and various $H$; see \cite{BH2026,CGHS2021,CamHoa2024,DhaHam2017,HoaMoo2015,KP2019,XJGH2025} and the references therein.
    In \cite{GS2018}, Goedgebeur and Schaudt gave an exhaustive algorithm to generate all $k$-vertex-critical $\mathcal{H}$-free graphs, where $\mathcal{H}$ is a set of graphs, and recently several researchers have used this algorithm to generate all $k$-vertex-critical $\mathcal{H}$-free graphs, for some specific $\mathcal{H}$ and when $k$ is small  (see \cite{XJGH2025} for an instance). On the other hand, for $k\ge 5$ and given a graph $H$, the only open case of \cref{prob-fin} for $H$-free graphs is when $H$ is $P_4$+$\ell P_1$ where $\ell\ge  1$, and is open even when $\ell=1$.  More precisely, we have the following problem.

    \begin{problem}[\cite{CHS2022}]\label{Prob-P4P1}For which values of $k \geq 5$, are there   finitely many $k$-vertex-critical ($P_4 + P_1$)-free graphs?
\end{problem}

The graph $P_4+P_1$ is  called the {\it co-gem}, and we refer to Figure~\ref{fig:named-graphs} for some named graphs on five vertices which are used in this paper. While \cref{Prob-P4P1} seems to be difficult in general,  Beaton and Cameron \cite{BC2025, BC2026}   studied  vertex-critical (co-gem, $H$)-free graphs, where $H$ is a graph on at most five vertices. If $H$ is any graph on at most four vertices, then they proved that for all $k\in \mathbb{N}$, there are finitely many $k$-vertex-critical (co-gem, $H$)-free graphs, and  they posed the following problem.
     \begin{problem}[\cite{BC2025}]\label{prob-cogem}
        For which  graphs $H$ of order five, are there finitely many $k$-vertex-critical (co-gem, $H$)-free graphs, for all $k\in \mathbb{N}$?
    \end{problem}

   \cref{prob-cogem} is resolved in the affirmative when $H\in\{$gem, co-dart, chair, banner$\}$. Indeed,  Abuadas et al.  \cite{ACHS2024} showed that for all $k\in \mathbb{N}$, there are  finitely many $k$-vertex-critical (co-gem, gem)-free graphs.   Beaton and Cameron \cite{BC2025,BC2026}    proved that for all $k\in \mathbb{N}$, there are finitely many $k$-vertex-critical (co-gem, $H$)-free graphs when $H\in \{$co-dart, chair$\}$.  However,  \cref{prob-cogem} is wide open for other graphs $H$ of order five.

   In  this paper, we identify three distinct graphs on five vertices that yield an affirmative  answer to \cref{prob-cogem}. More precisely,  we show that  for each $k\in \mathbb{N}$, there are finitely many $k$-vertex-critical (co-gem, $H$)-free graphs, where $H\in \{$paraglider, dart, house$\}$, by analysing the structure of such graphs. Our results together with a result of Couturier et al. \cite{CGKP2015} which states that $k$-{\sc Coloring}  can be solved in polynomial-time for co-gem-free graphs imply that for each $k\in \mathbb{N}$, there is a polynomial-time certifying algorithm for $k$-{\sc Coloring} of (co-gem, $H$)-free graphs, where $H\in \{$paraglider, dart, house$\}$.

\begin{figure}
\centering
 \includegraphics{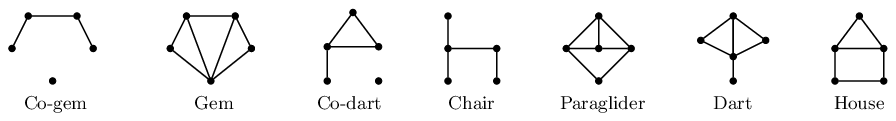}
 \caption{Some named graphs on five vertices.}\label{fig:named-graphs}
\end{figure}

\section{Preliminaries}
We use [$k$] to denote the set $\{1,2,\ldots,k\}$.
We use $P_n$ to denote  the chordless path on $n$ vertices, and  $K_n$  to denote the complete graph on $n$ vertices. For a graph $G$, co-$G$ (or $\overline{G}$) denotes the complement graph of $G$. A \emph{component} of a graph $G$ is a maximal connected induced subgraph of $G$. A graph $G$ is \emph{co-connected} if co-$G$ is connected. An \emph{anti-component} of a graph $G$ is a component of co-$G$. For $S\subset V(G)$, we use $G[S]$ to denote the subgraph induced by the set $S$, and  we use $G-S$ to denote the induced subgraph $G[V(G)\setminus S]$. For a vertex $v\in V(G)$, we write $G-v$ instead of $G-\{v\}$.
A \emph{clique cutset} of
a graph $G$ is a clique $Q$ in $G$ such that $G-Q$ has more components than $G$.

For   $v\in V(G)$,  a {\em neighbour} (resp., {\em nonneighbour})  of $v$ is a vertex which is adjacent (resp., nonadjacent) to $v$, and  $N(v)$ is the set of neighbours of $v$ in $G$.  For a subset $S\subseteq V(G)$, we let $N(S):= \{u\in V(G)\setminus S\mid u \text{ has a neighbour in } S\}$. Two nonadjacent vertices $u$ and $v$ in a graph $G$ are said to be \emph{comparable} if $N(u)\subseteq N(v)$ or $N(v)\subseteq N(u)$. For $S\subseteq V(G)$ and for two vertices $u, v\in V(G)\setminus S$, we say that $u$ and $v$ have the {\em same set of neighbours in $S$} if $N(u)\cap S = N(v)\cap S$.

Let  $v\in V(G)$ and let $S\subseteq V(G)\setminus \{v\}$. We say that \emph{$v$ is complete} (resp., \emph{anticomplete}) to $S$ if $v$ is adjacent to all (resp., none of) the vertices in $S$. We say that $v$ is \emph{mixed} on $S$, if $v$ has a neighbour and a nonneighbour in $S$.  Given a graph $G$ and two disjoint subsets $S_1$ and $S_2$ of $V(G)$, we say that $S_1$ is \emph{complete}  (resp., \emph{anticomplete}) to $S_2$ if each vertex in $S_1$ is  complete (resp., anticomplete) to $S_2$.

For an integer $p\geq 3$, we use $C_p$ to denote  the chordless cycle on $p$ vertices.
 A \emph{hole} is the graph $C_p$ for some $p\ge 5$. An \emph{odd hole} (resp., \emph{even hole}) is a hole with odd  (resp., even) number of vertices. For two vertex  disjoint graphs $G$ and $H$,  their \emph{union} is the graph with vertex-set $V(G)\cup V(H)$ and edge-set $E(G)\cup E(H)$, and is denoted by $G$+$H$. For  $r\in \mathbb{N}$, we use $rG$ to denote the graph consisting of the disjoint union of $r$ copies of $G$.

A \emph{module} $S$ of a graph $G$ is a subset of vertices such that every vertex in $V(G)\setminus S$ is either complete to $S$ or anticomplete to $S$. A \emph{homogeneous set}  is a module $S$ with $1 < |S| < |V(G)|$. A graph is   \emph{prime} if it does not contain a homogeneous set. Note that a prime graph with at least three vertices is connected and co-connected.

Let $G$ and $H$ be vertex-disjoint graphs, and let $v\in V(G)$. The graph obtained by \emph{substituting $H$ for $v$} has vertex-set $(V(G)\setminus \{v\})\cup V(H)$ and edge-set $E(G-v)\cup E(H)\cup \{ux\mid u\in V(H), x\in N(v)\}$. A \emph{blowup} of a graph $G$ is the graph obtained from $G$ by recursively substituting a nonempty complete graph for each vertex in $G$. A \emph{$k$-vertex-critical blowup} of a graph $G$  is a blowup of $G$ which is $k$-vertex-critical.

\medskip

We will also use the following well-known results.

\begin{rll}[Folklore]\label{lem:critical-clqcut-comp}
  A $k$-vertex-critical graph  does not contain a clique cutset or a pair of comparable vertices.
\end{rll}

\begin{rlt}[Strong Perfect Graph Theorem  \cite{spgt}]\label{spgt}
A graph
 is perfect if and only if it does not contain an odd-hole or the complement graph of an odd-hole.
\end{rlt}

\begin{rlt}[\cite{CGKP2015}] \label{kcol-ptime}
 For each $k\in \mathbb{N}$, $k$-{\sc Coloring}  can be solved in polynomial-time for co-gem-free graphs.
\end{rlt}

 \section{Vertex-critical (co-gem, paraglider)-free graphs}
 In this section, we show that for each $k\in \mathbb{N}$, there are finitely many $k$-vertex-critical (co-gem, paraglider)-free graphs by analysing the structure of such graphs.  First we prove an useful decomposition theorem for the class of $(P_4$, paraglider$)$-free graphs. To do the same, we need the following.

 For a graph $G$ and an integer $\ell\ge 2$,   we say that a partition ($W_0$, $W_1$, $W_2$, \ldots, $W_{\ell}$) of $V(G)$ is a \emph{good $\ell$-partition} if it satisfies the following:
\begin{enumerate}[label= $(\textsf{R}\arabic*)$, leftmargin=1.25cm] \itemsep=0pt \topsep=0pt
   \item\label{p4p-3} $W_0$ is a clique which is complete to $W_1\cup W_2$. Further if $G$ is disconnected, then $W_0=\emptyset$.

         \item\label{p4p-2} $|W_1|\ge 2$. Further if $G$ is connected, then $|W_2|\ge 2$.

 \item\label{p4p-1} For all $i\in [\ell]$,  $W_i$ is a nonempty set that induces a component of $G-W_0$.
    \item\label{p4p-4} For all $i\in [\ell]$, $N(W_i)\cap W_0$ is complete to $W_i$.

    \item\label{p4p-5} For all $i\in [\ell-1]$, $N(W_i)\cap W_0 \supseteq N(W_{i+1})\cap W_0$.
\end{enumerate}

Also, for a graph $G$ and an integer $\ell\ge 2$, we say that a good $\ell$-partition ($W_0$, $W_1$, $W_2$, \ldots, $W_{\ell}$) of $V(G)$ is a \emph{nice $\ell$-partition} if  there  is an index $\alpha\in [\ell]$ such that $|W_t|\geq 2$ for $t\in [\alpha]$, and $|W_{t}|=1$, for $t\in [\ell]\setminus [\alpha]$.

  \begin{lemma}\label{lem1-p4parg}
    Let $G$ be a  $(P_4$, paraglider$)$-free graph. Suppose that $G$  contains a $2P_2$. Then there is an integer $\ell \geq 2$ for which $G$ admits a good $\ell$-partition. Moreover, if $\ell \geq 2$ is the least integer for which $G$ admits a good $\ell$-partition, then  $G$ admits a nice  $\ell$-partition.
\end{lemma}
\begin{proof}
    Let $G$ be a $(P_4$, paraglider$)$-free graph. Suppose that $G$  contains a $2P_2$ induced by the vertex-set, say $\{a_{1}, a_{2}, a_3, a_4\}$ such that $\{a_1a_2, a_3a_4\}\subseteq E(G)$.
       First suppose that $G$ is disconnected. Let $G_1, G_2, \ldots, G_{\ell}$ be the components of $G$, where $|V(G_1)|\geq 2$ and $\ell\geq 2$. We may assume that $|V(G_1)|\geq |V(G_2)| \geq \cdots \geq |V(G_{\ell})|$. We let $W_i:= V(G_i)$, for  all $i\in [\ell]$, and let $W_0:=\emptyset$. Then clearly  ($W_0$, $W_1$, $W_2$, \ldots, $W_{\ell}$) is a nice  $\ell$-partition of $V(G)$, and we are done.
       So suppose that $G$ is connected. Now
    we choose $W_1$ and $W_2$ to be maximal disjoint subsets of $V(G)$, each of which induces a connected subgraph of $G$, such that
       $\{a_{1},a_2\}\subseteq W_1$,  $\{a_{3},a_4\}\subseteq W_2$, and   $W_1$ is anticomplete to $W_2$. Let $W_0:= \{v\in V(G)\setminus (W_1\cup W_2)\mid v \mbox{ has neighbours in both } W_1 \text{ and } W_2\}$.   By the maximality of $W_1$ and $W_2$,  no vertex in $V(G)\setminus (W_1\cup W_2)$   has a neighbor in exactly one of the sets $W_1$, $W_2$; so $N(W_1\cup W_2) = W_0$.
                To proceed further, we claim the following:
         \begin{claim}\label{p4p-w012}
         The set $W_0$ is complete to $W_1\cup W_2$, and $W_0$ is a clique.
         \end{claim}
   \noindent{\it Proof of \cref{p4p-w012}}.~First suppose to the contrary that there is a vertex, say $w_0\in W_0$ which is not complete to $W_1\cup W_2$. Then by the definition of $W_0$, for $i,~ j\in [2]$ and $i\neq j$, there are  vertices,  $w_i$, $w_i'\in W_i$ (say) and a vertex $w_j\in W_j$ (say) such that $w_iw_i', w_0w_i, w_0w_j\in E(G)$ and $w_0w_i'\notin E(G)$, and then $\{w_j,w_0,w_i,w_i'\}$ induces a $P_4$ which is a contradiction; so $W_0$ is complete to $W_1\cup W_2$.

    Next if there are nonadjacent vertices in $W_0$, say $w_0$ and $w_0'$, then since $W_0$ is complete to $W_1\cup W_2$ (by the first assertion), we see that $\{a_{1},w_0,a_{3},w_0',a_{2}\}$ induces a paraglider which is a contradiction; so $W_0$ is a clique. This proves \cref{p4p-w012}. $\sq$

     \begin{claim}\label{p4p-AB}
        For any two components of $G - (W_0\cup W_1\cup W_2)$ induced by (say) $A$ and $B$  respectively, we have $N(A)\cap W_0 \subseteq N(B)\cap W_0$ or $N(B)\cap W_0 \subseteq N(A)\cap W_0$. Moreover, $N(A)\cap W_0$ (resp., $N(B)\cap W_0$) is complete to $A$ (resp., $B$).
    \end{claim}
     \noindent{\it Proof of \cref{p4p-AB}}.~Assume the hypothesis. Now if there are vertices, say $w_a\in (N(A)\cap W_0)\setminus (N(B)\cap W_0)$ and $w_b\in (N(B)\cap W_0)\setminus (N(A)\cap W_0)$, then there are vertices, say $a\in A$ and $b\in B$, such that $aw_a, bw_b\in E(G)$, and then from \cref{p4p-w012}, $\{a, w_a, w_b, b\}$ induces a $P_4$ which is a contradiction. So the first assertion holds.

    Next, suppose that there is a vertex, say $w_a'\in N(A)\cap W_0$ which is not complete to $A$. Then there are adjacent vertices, say $a', a''\in A$, such that $a'w_a'\in E(G)$, but $a''w_a'\notin E(G)$. Then  from \cref{p4p-w012}, we see that $\{a'', a', w_a', a_1\}$  induces a $P_4$ which is a contradiction. So $N(A)\cap W_0$  is complete to $A$. The other case is similar. $\sq$

       \medskip

      By using \cref{p4p-w012} and \cref{p4p-AB}, for $i\in \{3,4,\ldots, \ell\}$ and $\ell\geq 3$, we let $W_i$ denote the vertex-set of a component  of $G- (W_0\cup W_1\cup W_2)$ such that $N(W_{i-1})\cap W_0 \supseteq N(W_{i})\cap W_0$.  Note that since $G$ is connected, for each $i\in [\ell]$,  $N(W_i)\cap W_0 \neq \emptyset$. Thus from above claims, we conclude that ($W_0$, $W_1$, $W_2$, \ldots, $W_{\ell}$) is a good $\ell$-partition of $V(G)$. This proves the first assertion of \cref{lem1-p4parg}.

\medskip
 Next suppose that $\ell \geq 2$ is the least  integer such that  $G$ admits a good $\ell$-partition, say ${\cal P}:=$~($W_0$, $W_1$, $W_2$, \ldots, $W_{\ell}$), of $V(G)$. Then we have the following:
     \begin{claim}\label{p4p-wi=1wj=1}
          If there is an $i\in \{3,4,\ldots, \ell-1\}$ such that $|W_i| = 1$ and $|W_{i+1}| \geq 2$, then $N(W_{i})\cap W_0$ = $N(W_{i+1})\cap W_0$.
    \end{claim}
    \noindent{\it Proof of \cref{p4p-wi=1wj=1}}.~Let $i\in \{3,4,\ldots, \ell-1\}$ be such that $|W_i| = 1$ and $|W_{i+1}| \geq 2$. Let $W_i:= \{w\}$, and let $w', w''\in W_{i+1}$ be adjacent.    Suppose to the contrary that $N(W_{i})\cap W_0 \neq  N(W_{i+1})\cap W_0$. Then since $N(W_{i})\cap W_0 \supseteq N(W_{i+1})\cap W_0$, there is a vertex, say $w_0\in (N(W_i)\cap W_0)\setminus (N(W_{i+1})\cap W_0)$. Now we let $W_0':=N(W_{i+1})\cap W_0$, $W_1':=W_1\cup W_2\cup\cdots\cup W_i\cup (W_0\setminus W_0')$, and $W_j':=W_{i+j-1}$ for $j\in \{2, 3, \ldots, \ell-i+1\}$. Then since $w_0\in W_0\setminus W_0'$, we see that $W_1'$ induces a component of $G-W_0'$ (by \ref{p4p-5}). Also by using \ref{p4p-3}-\ref{p4p-5}, it is easy  to verify  the following:  \begin{itemize}[label=$\circ$, leftmargin=0.75cm] \itemsep=0pt
    \item $W_0'$ is a clique  which is complete to $W_1'\cup W_2'$, and $W_0'\neq \emptyset$.
    \item $\{w_0,w\}\subseteq W_1'$ and  $\{w',w''\}\subseteq W_2'$.
        \item For all $r \in [\ell-i+1]$,   $W_r'$ is a nonempty set that  induces a component of $G-W_0'$.

        \item For all $j\in \{3,4,\ldots, \ell-i+1\}$, $N(W_j')\cap W_0'$ is complete to $W_j'$.
        \item For all $r\in [\ell-i+1]$, $N(W_r')\cap W_0' \supseteq N(W_{r+1}')\cap W_0'$.
                   \end{itemize}
 Thus we conclude that ($W_0'$, $W_1'$, $W_2'$, \ldots, $W_{\ell-i+1}'$) is a good ($\ell-i+1$)-partition of $V(G)$, and this  contradicts the choice of $\ell$, since  $\ell-i+1<\ell$ for $i\geq 3$. So $N(W_{i})\cap W_0$ = $N(W_{i+1})\cap W_0$. This proves \cref{p4p-wi=1wj=1}.  $\sq$

     \medskip
     Now if  ${\cal P}$  is a nice $\ell$-partition of $V(G)$, then we are done; otherwise there is an index $i\in \{3,4,\ldots,\ell-1\}$ such that $|W_i|$ = 1 and $|W_{i+1}|\ge 2$. Then we let ${\cal P}':=$~($W_0'$, $W_1'$, $W_2'$, \ldots, $W_i'$, $W_{i+1}'$, \ldots $W_{\ell}'$), where $W_j':=W_j$  for $j\in [\ell]\setminus \{i,i+1\}$, $W_i' := W_{i+1}$ and $W_{i+1}':=W_i$. Then by using \cref{p4p-wi=1wj=1}, we conclude that ${\cal P}'$ is again a good $\ell$-partition of $V(G)$.  If  ${\cal P}'$  is a nice $\ell$-partition of $V(G)$, then we are done; otherwise we continue the above process by applying \cref{p4p-wi=1wj=1} until we get a nice $\ell$-partition of $V(G)$. This proves the second assertion of \cref{lem1-p4parg}.
    \end{proof}

Next we explore the structure of a (co-gem, paraglider)-free graph that contains a $C_5$ and prove some useful properties.

\smallskip
  Let $G$ be a (co-gem, paraglider)-free graph. Suppose that $G$ contains a $C_5$, say with vertex-set $C:=\{v_1,v_2,\ldots, v_5\}$ and with edge-set
 $\{v_1v_2,v_2v_3,v_3v_4,v_4v_5,v_5v_1\}$. For $i\in [5]$ and $i$ mod $5$, we let:
 \begin{center}
\begin{tabular}{ccl}
$A_i$ &:= & $\{u \in V(G)\setminus C \mid N(u) \cap C = \{v_i, v_{i+1}\}\}$,\\
$B_i$  &:= &  $\{u \in V(G)\setminus C \mid N(u) \cap C = \{v_{i-1},v_{i+1}\}\}$,\\
$D_i$  &:= &  $\{u \in V(G)\setminus C \mid N(u) \cap C = \{v_{i-1},v_i,v_{i+1}\}\}$,\\

$X_i$  &:= & $\{u \in V(G)\setminus C\mid N(u) \cap C = \{v_i, v_{i+2}, v_{i-2}\}\}$,\\
$Y_i$  &:= & $\{u \in V(G)\setminus C\mid N(u) \cap C = C\setminus \{v_i\}\}$, and\\
$Z$ &:=& $\{u \in V(G)\setminus C \mid N(u) \cap C = C\}$.
\end{tabular}
\end{center}
We let $A:=\cup_{i=1}^5A_i$, $B:=\cup_{i=1}^5B_i$, $D:=\cup_{i=1}^5D_i$, $X:=\cup_{i=1}^5X_i$ and $Y:=\cup_{i=1}^5Y_i$. Then since $G$ does not contain a co-gem, we see that any $u\in V(G)\setminus C$ is adjacent to at least two vertices in $C$, and hence $N(C)= A\cup B\cup D\cup X\cup Y\cup Z$ and $V(G)=C\cup N(C)$. Moreover,  we claim the following:
\begin{enumerate}[label=  $(\textsf{L}\arabic*)$, leftmargin=1.5cm, series=edu*] \itemsep=0pt
\item \label{parg-A}\emph{For each $i\in [5]$ and $i$ mod $5$, the following hold:}
\begin{enumerate}[label=  ($\roman*$), topsep=0pt] \itemsep=0pt
 \item \emph{$A_i$ induces a $P_4$-free graph.}
 \item \emph{$A_i$ is complete to $(A\setminus A_i)\cup D_i\cup D_{i+1}$.}
 \item \emph{$A_i$ is anticomplete to $(B\setminus B_{i-2})\cup (X\setminus X_{i-2})\cup Y_i\cup Y_{i+1}$.}
\end{enumerate}
\end{enumerate}
\noindent{\it Proof of} \ref{parg-A}.~By symmetry, we prove the assertions for $i=1$.

\smallskip
\noindent{$(i)$}:~If $G[A_1]$ contains a  $P_4$, then $G[A_1 \cup \{v_{4}\}]$ contains  a co-gem; so  $G[A_1]$ is $P_4$-free. $\diamond$

\smallskip
\noindent{$(ii)$}:~If there are nonadjacent vertices, say $a\in A_1$ and $p\in (A \setminus A_1)\cup D_1\cup D_2$, then one of $\{v_5,v_4,v_3,p,a\}$ or $\{a,v_2,v_3,p, v_5\}$ or $\{a,v_1,v_5,p,v_3\}$ induces a co-gem; so $A_1$ is complete to $(A \setminus A_1)\cup D_1\cup D_2$. $\diamond$

\smallskip
\noindent{$(iii)$}:~Suppose that there are adjacent vertices, say $a\in A_1$ and $p\in (B\setminus B_{4})\cup (X\setminus X_{4})\cup Y_1\cup Y_{2}$. If $p\in B_1\cup X_2\cup X_5\cup Y_1$, then $\{v_1,v_2,p,v_5,a\}$ induces a paraglider, and if $p\in B_3$, then $\{p,a,v_1,v_5,v_3\}$ induces a co-gem, and the other cases are similar. $\sq$

\begin{enumerate}[label=  $(\textsf{L}\arabic*)$, leftmargin=1.5cm, resume=edu*] \itemsep=0pt
\item\label{parg-B}\emph{For each $i\in [5]$ and $i$ mod $5$, the following hold:} \begin{enumerate}[label=  ($\roman*$),topsep=0pt] \itemsep=0pt
\item \emph{$B_i$ is a stable set.}
\item \emph{No vertex in $B_{i-2}\cup B_{i+2}$ is mixed on $B_{i}$.}
\item \emph{$B_i$ is complete to $B_{i-1}\cup B_{i+1} \cup D_{i}\cup Y_{i-2}\cup Y_{i+2}\cup Z$.}
\item \emph{$B_i$ is anticomplete to $D_{i-1}\cup D_{i+1}\cup (X\setminus X_{i})\cup Y_{i-1}\cup Y_i\cup Y_{i+1}$.}
\end{enumerate}
\end{enumerate}
\noindent{\it Proof of} \ref{parg-B}.~By symmetry, we prove the assertions for $i=1$.

\smallskip
\noindent{$(i)$}:~If there are adjacent vertices in $B_1$, say $b_1$ and $b_1'$, then $\{v_1,v_2,b_1,v_5,b_1'\}$ induces a paraglider; so $B_1$ is a stable set. $\diamond$

\smallskip
\noindent{$(ii)$}:~Suppose to the contrary that there is a vertex, say $b_3\in B_3$, which is mixed on $B_1$. Then there are vertices, say  $b_1, b_1'\in B_1$, such that $b_3b_1\in E(G)$ and $b_3b_1'\notin E(G)$. But then $\{b_3,b_1,v_5,b_1',v_3\}$ induces a co-gem (by \ref{parg-B}:$(i)$); so no vertex in $B_3$ is mixed on $B_1$, and  the proof that no vertex in  $B_4$ is  mixed on $B_1$ is similar.  $\diamond$

\smallskip
\noindent{$(iii)$}:~Suppose to the contrary that there are nonadjacent vertices, say $b\in B_1$ and $p\in B_2\cup B_5\cup D_1\cup Y_3\cup Y_4\cup Z$. Now if $p\in B_2\cup B_5$, then $\{p,v_1,v_2,b,v_4\}$ or $\{p,v_1,v_5,b,v_3\}$ induces a co-gem, and if  $p\in  D_1\cup Y_3\cup Y_4\cup Z$, then $\{b,v_2,v_1,v_5,p\}$ induces a paraglider. $\diamond$

\smallskip
\noindent{$(iv)$}:~Suppose that there are adjacent vertices, say $b\in B_1$ and $q\in D_{2}\cup D_{5}\cup (X\setminus X_{1})\cup Y_{1}\cup Y_2\cup Y_{5}$, then $\{v_5,v_1,v_2,b,q\}$ induces a paraglider. $\sq$

\begin{enumerate}[label=  $(\textsf{L}\arabic*)$, leftmargin=1.5cm, resume=edu*] \itemsep=0pt
\item\label{parg-XY} {\em For each $i\in [5]$ and $i$ mod $5$, $|Y_i\cup X_{i-1}\cup X_{i+1}|\leq 1$.}
\end{enumerate}
\noindent{\it Proof of} \ref{parg-XY}.~If there are vertices, say $u$ and $w$, in $Y_i\cup X_{i-1}\cup X_{i+1}$, then one of $\{v_i,v_{i+1},u,v_{i-1},$ $w\}$ or $\{u,v_{i+2}, w,v_{i-1},v_{i+1}\}$ or  $\{u,v_{i+1}, w,v_{i-1},v_{i-2}\}$  induces a paraglider, or $\{u,v_{i+2},v_{i-2},w,$ $v_i\}$ induces a co-gem. So \ref{parg-XY} holds. $\sq$.

\begin{enumerate}[label=  $(\textsf{L}\arabic*)$, leftmargin=1.5cm, resume=edu*] \itemsep=0pt
\item\label{parg-B-com} \emph{For $i\in [5]$, if $|B_i|\geq 3$, then $G$ has a pair of comparable vertices.}
\end{enumerate}
 \noindent{\it Proof of} \ref{parg-B-com}.~We let $i=1$ and let $\{r,s,t\}\subseteq B_1$. We prove that there is a pair of comparable vertices in $\{r,s,t\}$.   Suppose to the contrary that no two vertices in $\{r,s,t\}$ are comparable. Then   since $|X_1|\leq 1$ (by \ref{parg-XY}), at least two vertices in  $\{r,s,t\}$ have the same set of neighbours in  $X_1$, say $r$ and $t$. Then since $r$ and $t$ are not comparable, from \ref{parg-A}:$(iii)$ and \ref{parg-B},  there are vertices, say $r'$ and $t'$ in $A_{3}\cup D_{3}\cup D_{4}$, such that $rr', tt'\in E(G)$ and $rt',tr'\notin E(G)$. Then one of $\{r,r',v_{3},t',v_1\}$ or $\{r,r',t',t,v_1\}$ induces a co-gem which is a contradiction. So \ref{parg-B-com} holds. $\sq$

\begin{enumerate}[label=  $(\textsf{L}\arabic*)$, leftmargin=1.5cm, resume=edu*] \itemsep=0pt
\item\label{parg-Z} \emph{$Z$ is a clique which is complete to $B\cup X\cup Y$.}
\end{enumerate}
 \noindent{\it Proof of} \ref{parg-Z}.~If there are nonadjacent vertices, say $z\in Z$ and  $p\in B_i\cup X_{i+1}\cup Y_i\cup Z$ for some $i\in [5]$, then $\{v_i,v_{i+1},p,v_{i-1},z\}$ or $\{v_i,z,v_{i-2},p,v_{i+1}\}$ induces a paraglider; so $Z$ is a clique which is complete to $B_i\cup X_{i+1}\cup Y_i$, and then \ref{parg-Z} follows since $i$ is arbitrary. $\sq$

\setcounter{claim}{0}

\begin{lemma}\label{lem:parg-A-2P2}
 Let $G$ be a $k$-vertex-critical (co-gem, paraglider)-free graph. Suppose that $G$ contains a $C_5$, say with vertex-set $C:=\{v_1,v_2,\ldots, v_5\}$ and with edge-set  $\{v_1v_2,v_2v_3,v_3v_4,v_4v_5,v_5v_1\}$. For $i\in [5]$, let $A_i$, $B_i$, $D_i$, $X_i$, $Y_i$  and $Z$ be defined as above. If $G[A_i]$ contains a $2P_2$, then $|A_i|$ is bounded above by a function of $k$, that is, $|A_i|\leq 3k+2^{3k+2}$.
\end{lemma}
\begin{proof}
       By symmetry, we may assume that $G[A_1]$ contains a $2P_2$. Then by using \ref{parg-A}:$(i)$ and \cref{lem1-p4parg}, we choose the least positive integer $\ell\geq 2$ such that  $G[A_1]$ admits a good $\ell$-partition. Since $\ell\geq 2$ is the least positive integer, again by \cref{lem1-p4parg},  $G[A_1]$ admits a nice $\ell$-partition, say ${\cal P}:=$ ($W_0$, $W_1$, $W_2$, \ldots, $W_{\ell}$), of $A_1$. Recall that,  by the definition of a nice $\ell$-partition,  there is an index $\alpha\in [\ell]$ such that $|W_t|\geq 2$ for $t\in [\alpha]$ and $|W_{t}|=1$, for $t\in [\ell]\setminus [\alpha]$. We let $W$ := $\cup_{j=1}^{\ell}W_j$.
   Now we claim the following, where $i\in [\ell]$.

   \begin{claim}\label{clm:A1-2p2-nei}
        Suppose that $|W_i|\ge 2$. Let $x\in N(W_i)\setminus (C\cup W_0)$.  Then $x\in D_1\cup D_2\cup X_4\cup Y_3\cup Y_4\cup Y_5\cup Z$, and hence $x$ is complete to $\{v_1,v_2\}$.  Moreover, if $x\notin Z$, then $x$ is complete to $W$.
    \end{claim}
    \noindent{\it Proof of \cref{clm:A1-2p2-nei}}.~Let $x\in N(W_i)\setminus (C\cup W_0)$ and let $x\in N(w_i)$ where $w_i \in W_i$.
Let $w_i'\in W_i$ be adjacent to $w_i$. Also, there is an index $j\in [\ell]$ and $j\neq i$ such that $|W_j|\neq \emptyset$, so let $w_j\in W_j$. Clearly, $x\notin A_1$ (by \ref{p4p-1}).  Suppose to the contrary that $x\notin D_1\cup D_2\cup X_4\cup Y_3\cup Y_4\cup Y_5\cup Z$.   So by \ref{parg-A}:$(ii)$ and  \ref{parg-A}:$(iii)$, we see that $x\in (A\setminus A_1)\cup B_4 \cup D_3\cup D_4\cup D_5$.  Now if $x\in A\setminus A_1$, then $\{v_1,w_i,w_i',w_j,x\}$  or $\{v_2,w_i,w_i',w_j,x\}$ induces a paraglider (by \ref{parg-A}:(ii)), and if $x\in B_4$, then $G[\{v_1,w_i,w_i',w_j,x, v_4\}]$ contains a co-gem or a paraglider; so $x\in D_3\cup D_4\cup D_5$. Also if $x\in D_3$, then since $G[\{v_5,v_4,x,w_i,w_i',w_j\}]$ does not contain a co-gem, we have $xw_i',xw_j\in E(G)$, and then $\{v_1,w_i,x,w_j,w_i'\}$ induces a paraglider; so $x\notin D_3$. Likewise, $x\notin D_5$. Hence $x\in D_4$. Then one of $\{x,w_i,v_2,v_3,w_i'\}$ or $\{v_1,w_i,x,w_j,v_2\}$ induces a paraglider, or $\{w_i',w_i,x,v_3,w_j\}$ induces a co-gem, a contradiction; so $x\in D_1\cup D_2\cup X_4\cup Y_3\cup Y_4\cup Y_5\cup Z$. In particular, $x$ is complete to $\{v_1,v_2\}$.

 Now we show that if $x\notin Z$, then $x$ is complete to $W$. If there is a vertex, say $w\in W$, which is nonadjacent to $x$, then from the first assertion and from \ref{parg-A}:$(ii)$, clearly $x\in X_4\cup Y_3\cup Y_4\cup Y_5$, and so $G[\{w_i, x, v_4, v_3, w_j, w\}]$ or $G[\{w_i, x, v_4, v_5, w_j, w\}]$ contains a co-gem; hence $x$ is  complete to $W$.   This proves \cref{clm:A1-2p2-nei}. $\sq$

  \begin{claim}\label{clm:A1-2p2-clq}
        Suppose that $|W_i|\ge 2$.  Then for any $x, y\in N(W_i)$,  we have  $xy\in E(G)$,  or   $x\in W_0$ and $y\in Z$ or vice versa.
    \end{claim}
   \noindent{\it Proof of \cref{clm:A1-2p2-clq}}.~Let $x$ and $y$  be two nonadjacent vertices in $N(W_i)$. From \cref{clm:A1-2p2-nei}, clearly neither $x$ nor $y$ is in $\{v_1, v_2\}$, and we may assume that  $x$, $y\in W_0\cup D_1\cup D_2\cup X_4\cup Y_3\cup Y_4\cup Y_5\cup Z$.  If $x$, $y\in W_0\cup D_1\cup D_2\cup X_4\cup Y_3\cup Y_4\cup Y_5$, then $\{x,y\}$ is complete to $W_1\cup W_2$ (by \cref{clm:A1-2p2-nei}), and then   $\{x, w_{1}, y, w_{2}, w_1'\}$ induces a paraglider, where $w_1$ and $w_1'$ are adjacent vertices in $W_1$  and  $w_2\in W_2$; so $x\in Z$ or $y\in Z$. We may assume that $y\in Z$. Then from \ref{parg-Z}, since $Z$ is a clique which is complete to $X_4\cup Y_3\cup Y_4\cup Y_5$, we have $x\in D_1\cup D_2\cup W_0$. Now if $x\in D_1$, then   for some $w_i\in W_i$ which is adjacent to $y$, we see that $\{v_2,x,v_5,y,w_i\}$   induces a paraglider (by \ref{parg-A}:$(ii)$); so $x\notin D_1$. Similarly, $x\notin D_2$. So we conclude that $x\in W_0$. This proves \cref{clm:A1-2p2-clq}. $\sq$

   \begin{claim}\label{clm:A1-2p2-cut}
  Suppose that $|W_i|\ge 2$. Then    there is a vertex in $N(W_i)\cap Z$ and there is a vertex in $N(W_i)\cap W_0$ which are nonadjacent.
   \end{claim}
      \noindent{\it Proof of \cref{clm:A1-2p2-cut}}.~If $N(W_i)\cap Z$ is complete to $N(W_i)\cap W_0$, then from \cref{clm:A1-2p2-clq}, $N(W_i)$ is a clique cutset separating the component containing $W_i$  and the component containing $\{v_3, v_4, v_5\}$ in $G-N(W_i)$ which is a contradiction to \cref{lem:critical-clqcut-comp}. So   there is a vertex in $N(W_i)\cap Z$ and there is a vertex in $N(W_i)\cap W_0$ which are nonadjacent. This proves \cref{clm:A1-2p2-cut}. $\sq$

\medskip
If  $|W_i|\ge 2$ for some $i\in [\ell]$, then by \cref{clm:A1-2p2-cut}, we let $z_i^*\in N(W_i)\cap Z$ and $q_i^*\in  N(W_i)\cap W_0$ be such that $z_i^*q_i^*\notin E(G)$. So since $|W_1|\geq 2$,    $z_1^*$ and $q_1^*$ exist and  $z_1^*q_1^*\notin E(G)$. Then since  $q_1^*\in W_0$, we have $W_0\neq \emptyset$, and hence $G[A_1]$ is connected (by \ref{p4p-3}). So from \ref{p4p-2}, we have $|W_2|\geq 2$, and hence   $z_2^*$ and $q_2^*$ exist and $z_2^*q_2^*\notin E(G)$. Moreover, since $G[A_1]$ is connected, we have $N(W_i)\cap W_0\neq \emptyset$, for all $i\in [\ell]$.

\begin{claim}\label{clm:A1-2p2-cut2}
  Suppose that $|W_i|\ge 2$ for some $i\in [\ell]$.  Then  $z_i^*$ is complete to $W_i$. Moreover, $z_i^*$ or $q_i^*$ is anticomplete to $W_j$, for all $j\in [\ell]$ and $j\neq i$. (So since $q_1^*,q_2^*\in W_0$ and since $\{q_1^*,q_2^*\}$ is complete to $W_1\cup W_2$, we have $z_1^*$ is anticomplete to $W_2$, and  $z_2^*$ is anticomplete to $W_1$.)
  \end{claim}
 \noindent{\it Proof of \cref{clm:A1-2p2-cut2}}.~Suppose to the contrary that $z_i^*$ is not complete to $W_i$. Since  $|W_1|\geq 2$ and $|W_2|\geq 2$, there is an index $j\in [2]$ and $j\neq i$ such that $|W_j|\geq 2$, and let $w_j,w_j'\in W_j$. Since $z_i^*$ is not complete to $W_i$,  there are adjacent vertices, say $w_i$, $w_i'\in W_i$, such that $w_iz_i^*\in E(G)$ and $w_i'z_i^*\notin E(G)$.   Then $z_i^*$ is complete to $\{w_j, w_j'\}$, for otherwise $G[\{w_i',w_i,z_i^*,v_4,w_j, w_j'\}]$ contains a co-gem, and then from \ref{p4p-3} and \ref{p4p-4}, $\{w_i, z_i^*, w_i', q_i^*, w_j\}$ induces a paraglider; so $z_i^*$ is complete to $W_i$.

    Next suppose that there  is an $r\in [\ell]$ and $r\neq i$  such that that $q_i^*$ is not anticomplete to $W_r$. If $z_i^*$ is adjacent to a vertex, say $w_r\in W_r$, then by using the first assertion and \ref{p4p-4}, we see that  $\{w_i, z_i^*, w_r, q_i^*, w_i'\}$ induces a paraglider; so  $z_i^*$  is anticomplete to $W_r$, for all $r\in [\ell]$ and $r\neq i$. This proves \cref{clm:A1-2p2-cut2}. $\sq$

\medskip

   \begin{claim}\label{clm:A1-2p2-Wi-clq}
        If $|W_i|\ge 2$, then $W_i$ is a clique.
    \end{claim}
    \noindent{\it Proof of \cref{clm:A1-2p2-Wi-clq}}.~Suppose to the contrary that there are nonadjacent vertices in $W_i$, say $w_i$ and $w_i'$. We show  that $w_i$ and $w_i'$  are comparable vertices. Suppose not. Then there is a vertex, say $x\in N(w_i)\setminus N(w_i')$. Clearly, from \cref{clm:A1-2p2-nei}  and \ref{p4p-4}, we have $x\in W_i\cup Z$. Also since  $|W_i|\ge 2$,  the vertices $z_i^*\in N(W_i)\cap Z$ and $q_i^*\in  N(W_i)\cap W_0$  exist and $z_i^*q_i^*\notin E(G)$, and hence   $z_i^*$ is complete to $W_i$ (by \cref{clm:A1-2p2-cut2}). Now if $x\in W_i$, then $\{q_i^*, w_i, z_i^*, w_i', x\}$ induces a paraglider; so $x\in Z$. Also  since $x$ is not complete to $W_i$, there are adjacent vertices, say $p_1$, $p_2\in W_i$, such that $xp_1\in E(G)$ and $xp_2\notin E(G)$, and hence if there is a vertex, say $w\in W\setminus W_i$, which is nonadjacent to $x$, then $\{p_2,p_1,x,v_4, w\}$ induces a co-gem; so $x$ is complete to $W\setminus W_i$. Then $xq_i^*\in E(G)$; for otherwise, for some adjacent vertices, say $w_j$, $w_j'\in W_j$, where $j\in \{1,2\}$ and $j\neq i$, $\{q_i^*, w_i, x, w_j, w_j'\}$ induces a paraglider (by \ref{p4p-3}). But now  $\{q_i^*, w_i', z_i^*, w_i, x\}$ induces a paraglider (by \ref{parg-Z}) which is a contradiction. So $w_i$ and $w_i'$  are comparable vertices in $G$ which is a contradiction to \cref{lem:critical-clqcut-comp}. So the claim holds. $\sq$

\begin{claim}\label{clm:A1-2p2-alpha2}
      For all $t\ge 3$, $|W_t|$ = 1; so we have $\alpha = 2$.
    \end{claim}
\noindent{\it Proof of  \cref{clm:A1-2p2-alpha2}}.~We will show that $|W_3|=1$. Suppose to the contrary that there are adjacent vertices, say $w_3$, $w_3'\in W_3$. Then for some $w_1\in W_1$ and $w_2\in W_2$,  since $\{w_1, z_1^*, z_2^*, w_2, w_3\}$ does not induce a co-gem (by \cref{clm:A1-2p2-cut2}), we have $z_1^*w_3\in E(G)$ or $z_2^*w_3\in E(G)$. So from \cref{clm:A1-2p2-cut2}, $q_1^*$ or $q_2^*$ is anticomplete to $W_3$. Due to symmetry, we may assume that $q_2^*$ is anticomplete to $W_3$. Now we let $W_0':= N(W_3)\cap W_0$, $W_1':=W_1\cup W_2\cup (W_0\setminus W_0')$, and $W_j':= W_{j+1}$ for $j\in \{2, 3, \ldots, \ell-1\}$. Then since $q_2^*\in (W_0\setminus W_0')$,  $W_1'$ induces a component of $G[A_1]-W_0'$ (by \ref{p4p-3}). Also by using \ref{p4p-3}-\ref{p4p-5}, it is easy to see that the following hold:
\begin{itemize}[label=$\circ$, leftmargin=0.6cm] \itemsep=0pt
 \item  $W_0'$ is a clique  which is complete to $W_1'\cup W_2'$, and $W_0'\neq \emptyset$.
 \item  $W_1\subseteq W_1'$ and  $\{w_3,w_3'\}\subseteq W_2'$.
        \item For all $t\in [\ell-1]$, $W_t'$ is a nonempty set that induces a component  of $G[A_1]-W_0'$.

        \item  For all $j\in \{3,4,\ldots, \ell-1\}$, $N(W_j')\cap W_0'$ is complete to $W_j'$.
       \item For all $t\in [\ell-1]$, $N(W_t')\cap W_0' \supseteq N(W_{t+1}')\cap W_0'$.
    \end{itemize}
 Thus we conclude that ($W_0'$, $W_1'$, $W_2'$, \ldots, $W_{\ell-1}'$) is a good ($\ell-1$)-partition of $A_1$, and this contradicts our choice of $\ell$. This proves \cref{clm:A1-2p2-alpha2}. $\sq$

\begin{claim}\label{clm:A1-2p2-Wi-1}
        For   $i \in \{3,4,\dots,\ell\}$,   we have  $N(W_i)\subseteq N(W_1)\cup D_4\cup X_4\cup Y_4\cup Z$. Moreover,  if there is a vertex, say  $d\in N(W_i)\cap D_4$, then  $d$ is complete to $D_4 \setminus \{d\}$; so $N(W)\cap D_4$ is a clique.
    \end{claim}
  \noindent{\it Proof of \cref{clm:A1-2p2-Wi-1}}.~Let  $i \in \{3,4,\dots,\ell\}$,  and let $W_i:=\{w_i\}$ (by \cref{clm:A1-2p2-alpha2}), and let $s\in N(w_i)$.  If $s\notin N(W_1)\cup D_4\cup X_4\cup Y_4\cup Z$, then by \ref{p4p-1}, \ref{parg-A} and \cref{clm:A1-2p2-nei}, we see that $s\in B_4\cup D_3\cup D_5\cup Y_3\cup Y_5$, and then for some $w_1\in W_1$, $\{w_i,s,v_3,v_4,w_1\}$  or $\{w_i,s,v_4,v_5,w_1\}$ induces a co-gem; so the first assertion holds.

  Next, suppose that $d\in N(w_i)\cap D_4$. If there is a vertex, say $d^*\in D_4\setminus \{d\}$ such that $dd^* \notin E(G)$, then from \cref{clm:A1-2p2-nei}, for some $w_1\in W_1$, $\{w_i,d,v_4,d^*, w_1\}$ induces a co-gem or $\{w_i,d,v_4,d^*,v_3\}$ induces a paraglider; so  $d$ is complete to $D_4 \setminus \{d\}$. Hence it follows from \cref{clm:A1-2p2-nei}, that $N(W)\cap D_4$ is a clique.  $\sq$

   \begin{claim}\label{A1-2p2-bndl}
        We have $\ell-2 ~\le~ 2^{|D_4\cup X_4\cup Y_4\cup Z\cup W_0|}~ \leq~ 2^{3k+2}$.
    \end{claim}
   \noindent{\it Proof of \cref{A1-2p2-bndl}}.~Suppose to the contrary that $\ell-2  >  2^{|D_4\cup X_4\cup Y_4\cup Z\cup W_0|}$. For $i, j\in \{3,4,\dots,\ell\}$ and $i\neq j$ and for any $x\in N(W_i)\setminus N(W_j)$, from \cref{clm:A1-2p2-nei} and \cref{clm:A1-2p2-Wi-1}, we have $x\in D_4\cup X_4\cup Y_4\cup Z\cup W_0$. However from \ref{p4p-3}, \ref{parg-Z}, \cref{clm:A1-2p2-Wi-1} and since $G$ is $k$-colorable, we have $|N(W)\cap (D_4\cup Z\cup W_0)|\le 3k$,  and from \ref{parg-XY}, we have $|X_4\cup Y_4|\le 2$. So if $\ell-2 > 2^{3k+2}$, there are indices $a, b\in\{3,4,\dots,\ell\}$ such that $N(W_a) = N(W_b)$. From \cref{clm:A1-2p2-alpha2}, we have $|W_a|$ = $|W_b|$ = 1; so the vertices in $W_a\cup W_b$ are a pair of comparable vertices which is a contradiction to \cref{lem:critical-clqcut-comp}. This proves \cref{A1-2p2-bndl}. $\sq$

 \medskip
    Now, from \cref{clm:A1-2p2-Wi-clq}, \cref{clm:A1-2p2-alpha2}, and \cref{A1-2p2-bndl}, we can conclude that $|A_1| = |W_0\cup W_1\cup W_2\cup \ldots \cup W_{\ell}| \le k+k+k+1+1+\ldots+1 = 3k+2^{3k+2}$. This completes the proof of \cref{lem:parg-A-2P2}.
\end{proof}

\begin{figure}
\centering
 \includegraphics{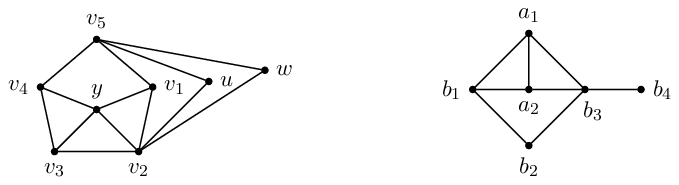}
 \caption{The graph  $F_1$ (left) and the graph $F_2$ (right).}\label{fig:F12}
\end{figure}

\begin{lemma}\label{parg-F2}
 If $G$ is  a   (co-gem, paraglider)-free graph that contains an $F_1$, then $G$ has a pair of comparable vertices.
\end{lemma}
\begin{proof}
Suppose that $G$ contains an $F_1$, say with vertices and edges as shown in Figure~\ref{fig:F12}. We let $C:= \{v_1,v_2,v_3,v_4,v_5\}$ and we partition $N(C)$ as earlier. Clearly $u,w \in B_1$ and $y\in Y_5$. Since $Y_5\neq \emptyset$, we have $X_1=\emptyset$ (by \ref{parg-XY}). Now we show that $u$ and $w$ are comparable. Suppose to the contrary that there is a vertex, say $u'$, such that $uu'\in E(G)$ and $wu'\notin E(G)$. Then by \ref{parg-A}:(iii) and \ref{parg-B}, we see that $u'\in A_{3}\cup D_{3}\cup D_{4}$.  But then one of $\{v_1,y,v_3,u',w\}$ or $\{v_1,y,u',u, w\}$ induces a co-gem which is a contradiction. So $u$ and $w$ are comparable vertices in $G$.  This proves \cref{parg-F2}.
\end{proof}

 \begin{theorem}\label{parg-main}
  For each $k\in \mathbb{N}$,  there are finitely many $k$-vertex-critical (co-gem, paraglider)-free graphs.
 \end{theorem}

\begin{proof} Let $G$ be a $k$-vertex-critical (co-gem, paraglider)-free graph. We show that $|V(G)|$ is bounded above by a function of $k$. We may assume that $G$ does not contain a $K_k$ and that $G$ is not perfect. Then   since an odd-hole $C_{2p+1}$ with $p\geq 3$ contains a co-gem and since the complement graph of an odd-hole co-$C_{2p+1}$ with $p\geq 3$ contains a paraglider, by \cref{spgt}, $G$ contains a $C_5$. We may assume such a  $C_5$, say with vertex-set $C:=\{v_1,v_2,\ldots, v_5\}$ and with edge-set
 $\{v_1v_2,v_2v_3,v_3v_4,v_4v_5,v_5v_1\}$. Then we partition $N(C)$ as earlier, and we use the structural properties \ref{parg-A} to \ref{parg-Z}. By \cref{lem:critical-clqcut-comp}, we may assume that $G$ does not contain a clique cutset or a pair of comparable vertices. From \ref{parg-XY}, \ref{parg-B-com}  and \ref{parg-Z}, we conclude that $|X|\leq 5$, $|Y|\leq 5$, $|B|\leq 10$   and $|Z|\leq k$. Thus, $|B\cup X\cup Y\cup Z| \leq k+20$. Next we show that  $|D|$ is bounded above by a function of $k$. To prove the same, we claim the following:

 \begin{claim1}\label{parg-D-A24}
Let $i\in [5]$ and $i$ mod $5$. For any two nonadjacent vertices, say $d$ and $d^*$, in $D_i$, no vertex in $A_{i+1}\cup A_{i+3}\cup (D_i\setminus \{d, d^*\}) \cup D_{i-1}\cup D_{i+1}$ is mixed on $\{d,d^*\}$.
 \end{claim1}
 \noindent{\it Proof of \cref{parg-D-A24}}.~We prove the claim when $i=1$. Suppose to the contrary that there is a vertex, say $p\in A_{2}\cup A_4\cup  (D_1\setminus \{d, d^*\}) \cup D_2\cup D_5$, which is mixed on $\{d,d^*\}$. We may assume that $pd\in E(G)$ and $pd^*\notin E(G)$. Then $\{d, p,v_3,v_4, d^*\}$ induces a co-gem or $\{d,v_2,d^*,v_5,p\}$ induces a paraglider which is a contradiction. So the claim holds. $\sq$

 \begin{claim1}\label{parg-Di-bound}
  For each $i\in [5]$, $|D_i|\leq k\cdot (2^{|B\cup X\cup Y\cup Z|+1}) \leq k\cdot 2^{k+21}$.
  \end{claim1}
 \noindent{\it Proof of \cref{parg-Di-bound}}.~We prove the claim when $i=1$. Since $G[D_1]$ is $k$-colorable, $D_1$ can be partitioned into $k$ stable sets, say $(U_1,U_2,\ldots, U_k)$. We will show that for each $j\in [k]$, $|U_j|\leq  2^{|B\cup X\cup Y\cup Z|+1}$. If $|U_j|>  2\cdot 2^{|B\cup X\cup Y\cup Z|}$ for some $j$, then by the pigeonhole principle, there are three nonadjacent vertices in $U_j$, say $u_1$, $u_2$ and $u_3$, that have the same set of neighbours in  $B\cup X\cup Y\cup Z$.  Since  $u_1$ and $u_2$ are not comparable, there are vertices $u_1'\in N(u_1)\setminus N(u_2)$ and $u_2'\in N(u_2)\setminus N(u_1)$. Clearly  from \ref{parg-A}:$(ii)$ and \cref{parg-D-A24}, we see that $u_1', u_2'\notin (A\setminus A_3) \cup D_1 \cup D_2\cup D_5$.   So we may assume that $u_1',u_2'\in A_3\cup D_3\cup D_4$. Then up to symmetry, we have the following three cases:
 \begin{itemize}[leftmargin=0.65cm] \itemsep=0pt \topsep=0pt
   \item  $u_1',u_2'\in A_3$: Now if $u_3u_1'\in E(G)$, then $\{v_1,u_1,u_1',u_3,v_2\}$ induces a paraglider; so $u_3u_1'\notin E(G)$. Likewise, $u_3u_2'\notin E(G)$. But then one of $\{u_1,u_1',v_4,u_2',u_3\}$ or $\{u_1,u_1',$ $u_2',u_2, u_3\}$ induces a co-gem which is a contradiction.

   \item  $u_1'\in A_3$ and $u_2'\in D_3$: Now from \ref{parg-A}:$(ii)$, we have  $u_1'u_2'\in E(G)$, and as in the previous item, we have $u_3u_1'\notin E(G)$. Also if $u_3u_2'\in E(G)$, then  $\{v_5,u_2,u_2',u_3, v_1\}$ induces a paraglider; so $u_3u_2'\notin E(G)$. But then $\{u_1,u_1',u_2',u_2, u_3\}$ induces a co-gem which is a contradiction.
   \item  $u_1'\in D_3$ and $u_2'\in D_3\cup D_4$:  As in the previous item, we have $u_3u_1'\notin E(G)$. But then $\{u_1,v_2,v_3,v_4,v_5,u_1',u_2,u_3\}$ induces an $F_1$, and hence from \cref{parg-F2}, $G$ has  a pair of comparable vertices which is a contradiction to our assumption that $G$ has no pair of comparable vertices.
 \end{itemize}
 This proves \cref{parg-Di-bound}. $\sq$

  \smallskip
  Thus from \cref{parg-Di-bound}, $|D|$ is bounded above by a function of $k$,  that is, $|D|\le 5k\cdot 2^{k+21}$. So it is enough to show that $|A|$ is bounded above by a function of $k$. First note that $|B\cup D\cup X\cup Y\cup Z| \leq  (5k\cdot 2^{k+21})+k+20$. We let  $\lambda:=(5k\cdot 2^{k+21})+k+20$.

  \begin{claim1}\label{parg-Ai-bound}
  For each $i\in [5]$,  $|A_i|\leq \max\{3k+2^{3k+2},~ 2^{\lambda}\}.$
  \end{claim1}
\noindent{\it Proof of \cref{parg-Di-bound}}.~We prove the claim when $i=1$.  Since $G[A_1]$ is $k$-colorable, $A_1$ can be partitioned into $k$ stable sets, say $(V_1,V_2,\ldots, V_k)$. We will show that  $|V_j|\leq   2^{|B\cup D\cup X\cup Y\cup Z|}$ for each $j\in [k]$  or $|A_1|\leq 3k+2^{3k+2}$. Suppose that $|V_j|>  2^{|B\cup D\cup X\cup Y\cup Z|}$ for some $j\in [k]$. Then by the pigeonhole principle, there are nonadjacent vertices in $V_j$, say $a_1$ and $a_2$, that have the same set of neighbours in  $B\cup D\cup X\cup Y\cup Z$.  Since  $a_1$ and $a_2$ are not comparable, there are vertices $a_1'\in N(a_1)\setminus N(a_2)$ and $a_2'\in N(a_2)\setminus N(a_1)$.  So from \ref{parg-A}:(ii), we have $a_1',a_2'\in A_1$. Then from \ref{parg-A}:$(i)$, since $G[A_1]$ is $P_4$-free, $\{a_1,a_1',a_2,a_2'\}$ induces a $2P_2$. Hence from \cref{lem:parg-A-2P2}, we have $|A_1|\leq 3k+2^{3k+2}$.  So the claim holds. $\sq$

 \smallskip
  Thus from \cref{parg-Ai-bound}, $|A|$ is bounded above by a function of $k$, and hence  $|V(G)|$ is  bounded above by a function of $k$. This completes the proof of \cref{parg-main}.
\end{proof}

From \cref{parg-main} and from \cref{kcol-ptime}, we have the following.

    \begin{corollary}
        For each $k\in \mathbb{N}$, there is a polynomial-time certifying algorithm for $k$-{\sc Coloring} of (co-gem, paraglider)-free graphs.
    \end{corollary}

\section{Vertex-critical (co-gem, dart)-free graphs}\label{sec:dart-free}

In this section, we show that for each $k\in \mathbb{N}$, there are finitely many $k$-vertex-critical (co-gem, dart)-free graphs by analysing the structure of such graphs. We use the following lemma which is a slight modification of Lemma 4 of \cite{XJGH2025}, and we give a proof here for completeness.

\setcounter{claim}{0}

\begin{lemma}\label{lem:dart-free}
   Let  $G$ be a dart-free graph. Let $S$ and $T$ be two disjoint subsets of $V(G)$ such that $|S|$ is bounded (say $|S|\leq c$ for some $c\in \mathbb{N}$), $S$ is not complete to $T$, and $G[T]$ is co-connected. Suppose that the following two conditions hold:
    \begin{enumerate}[topsep=0pt, noitemsep]
        \item[(I)]  For any two nonadjacent  vertices, say $x_1$ and $x_2$, in $T$, no vertex in $S$ is anticomplete to $\{x_1, x_2\}$.
        \item[(II)]  There exists a vertex, say $y\in V(G)\setminus (S\cup T)$, which is complete to $S\cup T$.
    \end{enumerate}
    Then $|T|$ is bounded above by a function of $\chi(G[T])$.
\end{lemma}
\begin{proof}
    Assume the hypotheses. We let $N_0 := S$, and for $i\ge 1$, we let $N_i:= \{u\in T\setminus (\bigcup_{j=0}^{i-1} N_j) \mid u \text{ has a nonneighbor in } N_{i-1}\}$. Since $S$ is not complete to $T$, clearly $N_1\neq \emptyset$. Then we claim the following.

   \begin{claim}\label{Nij}(a) For all $i,j\ge 1$, each $N_i$ is complete to $N_{j}$ for all $j\notin \{i-1,i+1\}$. (b)  There exists an integer $\ell$ such that $N_j = \emptyset$ for all $j\ge \ell$.
   \end{claim}

    \noindent{\it Proof of \ref{Nij}}:~Clearly $(a)$ follows from the definition of $N_i$'s. To prove $(b)$, we show that $\ell \le  2\chi(G[T])$. If $N_{2\chi(G[T])}\neq\emptyset$, then for each $i \leq  2\chi(G[T])$, pick a vertex, say $v_i\in N_i$, and then we see that $\{v_0, v_2, v_4, \ldots, v_{2\chi(G[T])}\}$ induces a clique of size $\chi(G[T])$+1 (by $(a)$), which is a contradiction. So there exists an integer $\ell \leq  2\chi(G[T])$ such that, for all $j\ge l$, $N_j$ = $\emptyset$.

 \begin{claim}\label{Nii+1} For all $i\ge 0$ and for any two nonadjacent vertices, say $z_1$ and $z_2$, in $N_{i+1}$, no vertex in $N_i$ is anticomplete to $\{z_1, z_2\}$.
 \end{claim}
     \noindent{\it Proof of \ref{Nii+1}}:~Clearly the claim is true for $i$ = 0 by our hypothesis (I). Suppose that the claim is not true for $i > 0$. So there is a vertex, say $v\in N_i$, which is anticomplete to $\{z_1,z_2\}$.  By the definition of $N_i$, there is a vertex, say $v' \in N_{i-1}$, which is nonadjacent to $v$. But then by using our hypothesis (II) and \ref{Nij}:$(a)$,  we see that $\{z_1,v',z_2,y,v\}$ induces a dart which is a contradiction. $\sq$

 \begin{claim}\label{Ni-bound} For all $i\ge 0$, we have $|N_i|\leq c~ (\chi(G[T]))^{i}$.
 \end{claim}
     \noindent{\it Proof of \ref{Ni-bound}}:~The proof is by induction on $i$. Since $|S|\le c$, the claim is true for $i = 0$. So we may assume that  $i\ge 1$. Note that since $\chi(G[N_i])\leq \chi(G[T])$, each $N_i\subseteq T$ can be partitioned into at most $\chi(G[T])$ stable sets. However, by \ref{Nii+1}, any stable set of $N_i$ can have at most one nonneighbor of $v$, where $v$ is any vertex in $N_{i-1}$. Thus, any stable set of $N_i$ can have at most $|N_{i-1}|$ vertices. This implies that $|N_i|\leq |N_{i-1}|~\chi(G[T])$, and hence $|N_i|\le c~(\chi(G[T]))^{i}$. This proves \cref{Ni-bound}.  $\sq$

\smallskip
  Now since $G[T]$ is co-connected, the conclusion of the lemma follows from the definition of $N_i$'s, $i\geq 1$ and by \ref{Nij}  and \ref{Ni-bound}.
    \end{proof}

We  first prove a few lemmas on the structure of  $($co-gem, dart$)$-free graphs.

\begin{lemma}\label{lem:str-prime-dart}
    If $G$ is a prime $($co-gem, dart$)$-free graph, then $G$ is $F_2$-free.
\end{lemma}
\begin{proof}
    Let $G$ be a prime (co-gem, dart)-free graph. Suppose to the contrary that $G$ contains an    $F_2$, say with vertices and edges as shown in Figure~\ref{fig:F12}. Since $\{a_1,a_2\}$ is not a homogeneous set, there is a vertex, say $v\in V(G)\setminus V(F_2)$, which is adjacent to $a_1$  but not to $a_2$. Then since $\{b_1,a_2,b_3,a_1,v\}$ does not induce a dart, $v$ must be adjacent to at least one of $b_1$ or $b_3$. So $vb_2\in E(G)$; for otherwise, $\{a_2,a_1,v,b_1,b_2\}$ or $\{a_2,a_1,v,b_3,b_2\}$ induces a dart, and hence $vb_4\in E(G)$; for otherwise $\{a_2,a_1,v,b_2,b_4\}$ induces a co-gem.    Now  if $vb_1\in E(G)$, then $\{a_1,b_1,b_2,v,b_4\}$ induces a dart, and if $vb_3\in E(G)$, then $\{b_2$, $v$, $b_4$, $b_3$, $a_2\}$  induces a dart which are  contradictions. Thus  we conclude that $G$ is $F_2$-free. This proves \cref{lem:str-prime-dart}.
\end{proof}

\setcounter{claim}{0}

\begin{lemma}\label{lem:antihole}
Let $G$ be a connected (gem, co-dart, $C_5$)-free graph. If $G$ contains a $C_{2p+1}$, where $p\ge 3$,  then $G$ is triangle-free.
\end{lemma}

\begin{proof}
Let $G$ be a connected (gem, co-dart, $C_5$)-free graph. Suppose that $p$ is the least integer such that $G$ contains an odd hole of length $2p+1$ where $p\ge 3$,  say with vertex-set $C:=\{v_1,v_2,\dots, v_{2p+1}\}$ and edge-set $\{v_1v_2,v_2v_3, \ldots, v_{2p}v_{2p+1}, v_{2p+1}v_1\}$.  Then we claim the following.

 \begin{claim}\label{dart-c7-no2c}
  For all $i\in [2p+1]$ and $i$ mod $2p+1$,  $N(v_i)\cap N(v_{i+1})$ is empty.
  \end{claim}
\noindent{\it Proof of \cref{dart-c7-no2c}}.~By  symmetry, we let $i = 1$. Suppose to the contrary that there is a vertex in  $N(v_1)\cap N(v_2)$, say $x$. Then since $\{v_{2p+1},v_1,v_2,v_3,x\}$ does not induce a gem,  $x$ is nonadjacent to $v_{2p+1}$ or $v_3$. Again by symmetry, we may assume that   $xv_{2p+1}\notin E(G)$. Now if $x$ is nonadjacent to $v_{2p-1}$ (resp., $v_{2p-2}$), then $\{x,v_2,v_1,v_{2p+1},v_{2p-1}\}$ (resp., $\{x,v_2,v_1,v_{2p+1},v_{2p-2}\}$) induces a co-dart; so
$xv_{2p-1}\in E(G)$ and $xv_{2p-2}\in E(G)$. But then $\{v_{2p-1},v_{2p-2},x,v_2,v_{2p+1}\}$ induces a co-dart which is a contradiction. So \cref{dart-c7-no2c} holds.  $\sq$

\begin{claim}\label{dart-c7-2nc}If $x\in N(C)$, then $N(x) \cap C \subseteq \{v_j, v_{j+2}\}$ for some $j \in [2p+1]$ and $j$ mod $2p+1$.
\end{claim}
  \noindent{\it Proof of \cref{dart-c7-2nc}}.~Let $x\in N(C)$. Suppose that $x$ has at least two neighbours in $C$, and we let $xv_1\in E(G)$.  From \cref{dart-c7-no2c}, $x$ is not adjacent to two consecutive vertices of $C$.  Note that for each $t$, $G[C]-v_t$ is an induced path with even number of vertices. Thus, if there is a  vertex in $N(x)\cap C$, say $v_k\neq v_1$  (if exists), then $G[C]-\{v_1,v_k\}$ is the union of two disjoint paths of different parity on the number of vertices. So we can conclude that, $G[C\setminus N(x)]$ always has an induced path with even number of vertices, say $P$. If $N(x)\notin \{\{v_1\}$, $\{v_{2p}, v_1\}, \{v_1, v_{3}\}\}$, then $P\cup \{x\}$ induces an odd hole of length less than $2p+1$ which is a contradiction to our choice of $p$ or $P\cup \{x\}$ induces a $C_5$ which is a contradiction to the fact that $G$ is $C_5$-free. So \cref{dart-c7-2nc} holds. $\sq$

\begin{claim}\label{dart-c7-k3free} $G[C\cup N(C)]$ is triangle-free.
\end{claim}
\noindent{\it Proof of \cref{dart-c7-k3free}}.~Suppose to the contrary that  $G[C\cup N(C)]$ contains a triangle. We choose a triangle  in $G$, say $K$, such that $|C\cap V(K)|$ is minimum.  Let $V(K):=T_0:= \{x,y,z\}$. Note that $|C\cap T_0|\in \{0,1\}$, by \cref{dart-c7-no2c}.
 If $C\cap T_0\neq \emptyset$, say $x:= v_1$ and $y,z \in N(C)$, then  from \cref{dart-c7-2nc}, $\{y,z,x,v_{2p+1}, v_4\}$ induces a co-dart; so $T_0 \subseteq N(C)$, and no two vertices in $T_0$ have a common neighbour in $C$ (by our choice of $K$). From \cref{dart-c7-2nc},  there is a vertex, say $u$, in $C$, which is anticomplete to $\{x,y,z\}$. Moreover, since   no two vertices in $T_0$ have a common neighbour in $C$, there is a vertex in $N(T_0)\cap C$, say $w$, which is nonadjacent to $u$. Now $\{x,y,z,w,u\}$ induces a co-dart which is a contradiction. So \cref{dart-c7-k3free} holds. $\sq$

\smallskip
    To proceed further, we let $N_0 := C$, and for each $i\in \mathbb{N}$, we let $N_i :=  \{v\in V(G)\setminus (\cup_{j=0}^{i-1}~ N_j) \mid  v $ has a neighbour in $ N_{i-1}\}$. Note that $N_1:= N(C)$. Then we prove the following.

\begin{claim}\label{dart-c7-nind}For all $j\ge 0$ and for each $v\in N_j$, $N(v)$ is a stable set.
\end{claim}
 \noindent{\it Proof of \cref{dart-c7-nind}}.~The proof is by induction on $j$. From \cref{dart-c7-k3free}, the claim holds when $j = 0$. Suppose to the contrary that the claim is not true for $j=1$.  Then there is a vertex, say $v\in N_1$, and there are two adjacent vertices, say $u$ and $w$, in $N(v)$. Then from \cref{dart-c7-k3free}, $u$ or $w$ must be in $N_2$. We may assume that $w\in N_2$ and $v\in N(v_1)$. Now if $uv_4\notin E(G)$, then from \cref{dart-c7-2nc}, $\{u,w,v,v_1,v_4\}$ induces a co-dart, and if $uv_4\in E(G)$, then again from \cref{dart-c7-2nc}, $\{v_4,u,v,w,v_{2p+1}\}$ induces a co-dart each of which is a contradiction. Hence the claim is true for $j = 1$. If possible, let $\ell\ge 2$ be the smallest integer for which the claim is not true. Then there is a vertex, say $u_1\in N_{\ell}$, such that there are vertices, say $u_2$, $u_3\in N(u_{1})$, which are adjacent. Let $w_1 \in N_{\ell-1}$ be a neighbour of $u_1$. Then by the choice of $\ell$, $w_1$ is anticomplete to $\{u_2, u_3\}$. Since $\ell\ge 2$, from \cref{dart-c7-2nc}, there is a vertex, say $w_2 \in C$, which is nonadjacent to $w_1$. Then  $\{u_2,u_3,u_1,w_1,w_2\}$ induces a co-dart which is a contradiction. So  \cref{dart-c7-nind} holds. $\sq$

\smallskip
    Since $G$ is connected, the result follows from \cref{dart-c7-nind}.
\end{proof}

\begin{lemma}\label{lem:contain_C5}
    Let $G$ be a  prime  $($co-gem, dart$)$-free graph such that $\chi(G)\leq k$. Suppose that $G$ contains a $C_5$. Then $|V(G)|$ is bounded above by a function of $k$.
\end{lemma}
\begin{proof}
     Let $G$ be a  prime  $($co-gem, dart$)$-free graph such that $\chi(G)\leq k$. Suppose that $G$ contains a $C_5$, say with vertex-set $C:=\{v_1,v_2,\ldots, v_5\}$ and with edge-set  $\{v_1v_2,v_2v_3,v_3v_4,v_4v_5,v_5v_1\}$.    For $i\in [5]$ and $i$ mod $5$, we define the following the sets:
\begin{center}
\begin{tabular}{ccl}
$V_i$  &:= &  $\{u \in V(G) \mid N(u) \cap C = \{v_{i-1}, v_{i+1}\}\text{ or } \{v_{i-1},v_i,v_{i+1}\}\}$,\\
$A_i$ &:= & $\{u \in V(G)\setminus C \mid N(u) \cap C = \{v_i, v_{i+1}\}\}$,\\
$B_i$  &:= & $\{u \in V(G)\setminus C\mid N(u) \cap C = \{v_i, v_{i+1}, v_{i-2}\}\}$,\\
$D_i$  &:= & $\{u \in V(G)\setminus C\mid N(u) \cap C = C\setminus \{v_{i+2}\}\}$, and\\
$T$ &:=& $\{u \in V(G)\setminus C \mid N(u) \cap C = C\}$.
\end{tabular}
\end{center}
Note that $v_i\in V_i$ for  each $i\in [5]$. Let $A:= \cup_{i=1}^{5} A_i$, $B:=\cup_{i=1}^{5} B_i$, and $D:= \cup_{i=1}^{5} D_i$. Then since $G$ does not contain a co-gem, it is not difficult to see that  $V(G)$ = ($\cup_{i=1}^{5} V_i) \cup A\cup B\cup D \cup T$. We  first prove some useful structural properties, where $i\in [5]$ and $i$ mod $5$.
\begin{claim}\label{dart-c5-Vi} For each $i\in [5]$, $V_i$ is   a stable set or a clique.
\end{claim}
\noindent{\it Proof of} \cref{dart-c5-Vi}:~It is enough to show that $G[V_i]$ is ($P_3$, $P_2+P_1$)-free. If $G[V_i]$ contains a $P_3$, say with vertex-set $W$, then $W\cup \{v_{i+1},v_{i+2}\}$ induces a dart; so $G[V_i]$ is $P_3$-free. Next  if $G[V_i]$ contains a  $P_2+P_1$, say with vertex-set $W'$, then $W'\cup \{v_{i-1}, v_{i+1}, v_{i+2}\}$ induces an $F_2$ which is a contradiction to Lemma~\ref{lem:str-prime-dart}; so   $G[V_i]$ is ($P_2+P_1$)-free. This proves \ref{dart-c5-Vi}. $\sq$

\begin{claim}\label{dart-c5-Vistable} If there is an $i\in [5]$ such that $V_i$ is a stable set,   then $V_i$ is complete to $V_{i-1}\cup V_{i+1}$ and is anticomplete to $V_{i-2}\cup V_{i+2}$.
\end{claim}
 \noindent{\it Proof of} \cref{dart-c5-Vistable}:~Suppose that $V_1$ is a stable set.  By symmetry, it is sufficient to prove that $V_1$ is  complete to $V_2$, and is anticomplete to $V_3$.    Let $u\in V_1$ be arbitrary.   Now if there is a vertex in $V_2$, say $x$, which is nonadjacent to $u$, then $u\neq v_1$, and so $\{u,v_2,v_1,x,v_4\}$ induces a co-gem or $\{v_1,x,v_3,v_2,u\}$ induces a dart; so $u$ is complete to $V_2$. Also  if  there is a vertex in $V_3$, say $y$, which is adjacent to $u$, then $u\neq v_1$, and thus  $\{u,y,v_4,v_3,v_1\}$ induces a co-gem or $\{u,y,v_3,v_2,v_1\}$ induces a dart; so $u$ is anticomplete to $V_3$. So $V_1$ is  complete to $V_2$, and is anticomplete to $V_3$, since $u$ is arbitrary.  $\sq$

\begin{claim}\label{dart-c5-ABD}  For each $i\in [5]$, $A_i\cup B_i\cup D_i$ is a clique, and so $|A_i\cup B_i\cup D_i|\leq k$.
\end{claim}
  \noindent{\it Proof of} \cref{dart-c5-ABD}:~If there are nonadjacent vertices in $A_i\cup B_i\cup D_i$, say $x$ and $y$,   then $\{x,v_{i},y,v_{i+1},$ $ v_{i+2}\}$ induces a dart; so \ref{dart-c5-ABD} holds. $\sq$

\begin{claim}
   \label{dart-c5-U} For each $i\in [5]$, $A_i$ is anticomplete to $T$.
    \end{claim}

\noindent{\it Proof of} \cref{dart-c5-U}:~If there are adjacent vertices, say $a\in A_i$ and $t\in T$, then $\{a,v_{i+1},v_{i+2},t,$ $ v_{i-1}\}$ induces a dart; so \ref{dart-c5-U} holds. $\sq$

\begin{claim}\label{dart-c5-UVB}  If $T\neq \emptyset$, then $B$ is an empty set, and  $V_i$ is a clique, for all $i\in [5]$.
\end{claim}
  \noindent{\it Proof of} \cref{dart-c5-UVB}:~Let $t\in T$. Suppose that there is a vertex, say $x\in V_i \cup B_{i+1}$, where $i\in [5]$.  Then since $\{v_{i},t,v_{i-2},v_{i-1},x\}$ or $\{v_i, v_{i+1},x,t, v_{i-2}\}$ does not induce a dart, we have $x\notin B_{i+1}$ and $xv_i\in E(G)$. So we conclude that $B_{i+1}=\emptyset$, and from \cref{dart-c5-Vi}, $V_i$ is a clique. So \cref{dart-c5-UVB} holds, since $i$ is arbitrary.  $\sq$

\begin{claim}\label{dart-c5-compare} If there are two nonadjacent vertices in $T$, say $x$ and $y$, then any vertex in $V_i\cup D_{i-2}$ is adjacent to at least one of $x$ and $y$.
\end{claim}
   \noindent{\it Proof of} \cref{dart-c5-compare}:~If there is a vertex in $V_i\cup D_{i-2}$, say $z$, which is anticomplete to $\{x,y\}$, then $z\neq v_i$, and so   $\{x,v_{i+1},y,v_{i+2},z\}$ or $\{x,v_i,y,v_{i+1},z\}$ induces a dart; so \cref{dart-c5-compare} holds.  $\sq$

\medskip
From \ref{dart-c5-ABD},  we have $|A\cup B\cup D| \le  5k$. Now we will show that $|V(G)|$ is bounded  above by a function of $k$ in two cases based on the set $T$.

\medskip
    \noindent
    {\textbf{Case 1}}.~\emph{$T$ is empty.}

\smallskip
    Then $V(G)$ = $(\bigcup_{i=1}^{5} V_i) ~\cup A \cup B \cup D$. It is sufficient to prove that $|V_i|$ is bounded above by a function of $k$, for each $i\in [5]$. Let $i\in [5]$ be arbitrary. From \ref{dart-c5-Vi}, $V_i$ is  a clique or a stable set. Obviously, if $V_i$ is a clique, then $|V_i|\le k$ and we are done; so we may assume that $V_i$ is a stable set.    From \ref{dart-c5-Vistable}, no vertex in  $(\bigcup_{j=1}^{5} V_j)\setminus V_i$ is mixed on two vertices of $V_i$. Thus if $|V_i|>  2^{5k}$, then since $|A\cup B\cup D| \le  5k$,  by the pigeonhole principle, there are two vertices, say $p$ and $q$, in $V_i$ such that $p$ and $q$ have the same set of neighbours in $A\cup B\cup D$. This implies that  $\{p,q\}$ is a homogeneous set in $G$ (by \ref{dart-c5-Vistable}) which is a contradiction to our assumption  that $G$ is prime; so
     $|V_i|\leq  2^{5k}$, and we are done.

    \medskip
    \noindent
 {\textbf{Case 2}}.~\emph{$T$ is nonempty.}

\smallskip
    Then $V(G)$ = $(\bigcup_{i=1}^{5} V_i) ~\cup A \cup D \cup T$, by \ref{dart-c5-UVB}. Also from \ref{dart-c5-ABD}  and \ref{dart-c5-UVB},   $A_i\cup D_i$ and $V_i$ are cliques for each $i\in [5]$; so we have $|V(G)\setminus T|\leq 5k+5k = 10k$. Thus, it is sufficient to prove that $|T|$ is bounded above by a function of $k$. Note that, since $G[T]$ is $(k-3)$-colorable, there are at most ($k-3$) anti-components in $G[T]$. So we will show that any arbitrary anti-component, say induced by $U$,  of $G[T]$  has bounded number of vertices. Recall that from \ref{dart-c5-U}, $A$ is anticomplete to $T$. If for all $i\in [5]$, $i$ mod $5$, no vertex in $V_i \cup D_{i-2}$ is mixed on $U$, then $U$ is a homogeneous set in $G$ when $|U|\ge 2$, and thus $|U|$ = 1 and we are done. So we   may assume that there is an $i\in [5]$, $i$ mod $5$ and there is a vertex, say $s$, in $V_i\cup D_{i-2}$,   which is mixed on $U$. Now we let $S:=\{s\}$ and   $y:= v_{i+1}$. Then since $s$ is mixed on $U$ and $G[U]$ is co-connected,  from \ref{dart-c5-compare} and from Lemma \ref{lem:dart-free}, it follows that $|U|$ is bounded above by a function of $k$.

    This completes the proof of \cref{lem:contain_C5}.
\end{proof}

We will also use the following theorem.

\begin{rlt}[\cite{BH2026}]\label{thm:clq-sklton}
    Let $\mathcal{G}$ be a hereditary class of graphs, and let $k\in \mathbb{N}$. For each $\ell \in [k]$, suppose that the number of vertices in each $\ell$-vertex-critical blowup of every prime graph in $\cal G$ is bounded above by a function of $\ell$. Then for each $p\in [k]$, the number of vertices in each $p$-vertex-critical graph in $\cal G$ is bounded above by a function of $p$.
\end{rlt}

We are now ready to prove the main theorem of this section, and is given below.

 \begin{theorem}\label{thm:critical-dart-free}
       For each $k\in \mathbb{N}$, there are finitely many $k$-vertex-critical $($co-gem, dart$)$-free graphs.
\end{theorem}

\begin{proof}
   Let $\mathcal{G}$ be the class of $($co-gem, dart$)$-free graphs, and let $k\in \mathbb{N}$.   We prove that   the number of vertices in a $k$-vertex-critical  graph in $\cal G$ is bounded above by a function of $k$. From Theorem \ref{thm:clq-sklton}, it is sufficient to prove that the number of vertices in each $\ell$-vertex-critical blowup of every prime graph in $\mathcal{G}$ is bounded above by a function of $\ell$, for each $\ell\in [k]$.

    Let $G\in \cal G$ be a prime graph, and let $\ell\in [k]$.  Consider an arbitrary $\ell$-vertex-critical blowup of $G$, say $H$. If $G$ is perfect, then $H$ is the complete graph $K_{\ell}$, and hence $H$ has $\ell$ vertices.   So we may assume that $G$ is not perfect. Then since $G$ is co-gem-free, by \cref{spgt}, $G$  must contain a   co-$C_{2p+1}$, for some $p\ge 2$. First suppose that $G$ contains a $C_5$.
    Since $G$ is an induced subgraph of  $H$,   $G$ must be $\ell$-colorable. So from Lemma \ref{lem:contain_C5}, $|V(G)|$ is bounded above by a function of $\ell$.
     Then since  $H$ can be obtained from $G$ by substituting a clique of size at most $\ell$  for each vertex of $G$,   $|V(H)|$ is bounded above  by a function of $\ell$, and we are done.   So we may assume that $G$ contains a co-$C_{2p+1}$, where $p\ge 3$. Since $G$ is prime, $G$ is co-connected.  Then from   \cref{lem:antihole}, it follows that $G$ is 3$K_1$-free. Since any blowup of a 3$K_1$-free graph is also 3$K_1$-free, we see that $|V(H)|$ is bounded above by the Ramsey number $R(3,\ell+1)$, and we  are done.
\end{proof}

From \cref{thm:critical-dart-free} and from \cref{kcol-ptime}, we have the following.

    \begin{corollary}
        For each $k\in \mathbb{N}$, there is a polynomial-time certifying algorithm  for $k$-{\sc Coloring} of (co-gem, dart)-free graphs.
    \end{corollary}

    \section{Vertex-critical (co-gem, house)-free graphs}\label{sec:house-free}
In this section, we show that for each $k\in \mathbb{N}$, there are finitely many $k$-vertex-critical (co-gem, house)-free graphs by using a structure theorem  which  directly follows from a sequence of lemmas proved by  Brandst\"adt and Kratsch in \cite{Brandstadt2005}, and is given below.

\begin{rlt}[\cite{Brandstadt2005}]\label{thm:str-housecogem}
   Let $G$ be a prime (co-gem, house)-free graph. Then   $G$ is perfect, or
            $G$ has at most nine vertices, or $G$ is $3K_1$-free, or $G$ contains a clique cutset.
        \end{rlt}

%

\begin{theorem}\label{thm:critical-house-free}
       For each $k\in \mathbb{N}$, there are finitely many $k$-vertex-critical (co-gem, house)-free graphs.
\end{theorem}

\begin{proof}[Proof of Theorem \ref{thm:critical-house-free}]
    Let $\mathcal{G}$ be the class of $($co-gem, house$)$-free graphs, and let $k\in \mathbb{N}$.   We prove that   the number of vertices in a $k$-vertex-critical  graph in $\cal G$ is bounded above by a function of $k$. From Theorem \ref{thm:clq-sklton}, it is sufficient to prove that the number of vertices in each $\ell$-vertex-critical blowup of every prime graph in $\mathcal{G}$ is bounded above by a function of $\ell$, for each $\ell\in [k]$.

    Let $G\in \cal G$ be a prime graph, and let $\ell\in [k]$.  Consider an arbitrary $\ell$-vertex-critical blowup of $G$, say $H$. If $G$ has a clique cutset, then $H$ also has a clique cutset which is a contradiction to \cref{lem:critical-clqcut-comp}; so we may assume that $G$ does not contain a  clique cutset. Now we apply \cref{thm:str-housecogem}. If $G$ is perfect, then  $H$ is the complete graph $K_{\ell}$ and hence $H$ has $\ell$ vertices,  and if $G$ has  at most nine vertices, then $H$ has at most $9\ell$ vertices, and we are done. Next if $G$ is 3$K_1$-free, then since any blowup of a 3$K_1$-free graph is also 3$K_1$-free, the number of vertices in $H$ is bounded above by the Ramsey number $R(3,\ell+1)$.  This proves \cref{thm:critical-house-free}.
\end{proof}

From \cref{thm:critical-house-free} and from \cref{kcol-ptime}, we have the following.

    \begin{corollary}
        For each $k\in \mathbb{N}$, there is a polynomial-time certifying algorithm  for $k$-{\sc Coloring} of (co-gem, house)-free graphs.
    \end{corollary}

\medskip
\noindent{\bf Acknowledgement}.~The authors sincerely thank  Ch\'inh T. Ho\`ang for the  initial  discussions, and Ben Cameron for his comments and suggestions on the preliminary version of this paper.

\end{document}